\input ./style/arxiv-vmsta.cfg
\documentclass[numbers,compress,v1.0.1]{vmsta}

\volume{3}
\issue{2}
\pubyear{2016}
\firstpage{107}
\lastpage{117}
\doi{10.15559/16-VMSTA54}% Updated by VTEXPTS2LaTeX.exe, 03.06.2016
\startlocaldefs

\def\arxivurl#1{\href{http://www.arxiv.org/abs/#1}{arXiv:#1}}
\newtheorem{thm}{Theorem}
\newtheorem{lemma}{Lemma}

\theoremstyle{definition}

\newtheorem{remark}{Remark}
\allowdisplaybreaks
\endlocaldefs

\begin{document}
\begin{frontmatter}
\title{Large deviations for drift parameter estimator of~mixed
fractional Ornstein--Uhlenbeck process}
\author{\inits{D.}\fnm{Dmytro}\snm{Marushkevych}}\email{dmmar1992@gmail.com}
\address{Laboratoire Manceau de Math\'ematiques, Facult\'e des Sciences
et Techniques, Universit\'e du Maine, Avenue Olivier Messiaen, 72085,
Le Mans, France}
\markboth{D. Marushkevych}{Large deviations for mixed fractional
Ornstein\xch{--}{-}Uhlenbeck process}
\begin{abstract}
We investigate large deviation properties of the maximum likelihood
drift parameter estimator for Ornstein--Uhlenbeck process driven by
mixed fractional Brownian motion.
\end{abstract}
\begin{keyword}
Large deviations\sep
Ornstein--Uhlenbeck process\sep
mixed fractional Brownian motion\sep
maximum likelihood estimator
\MSC[2010] 60G15\sep62F12\sep60G22
\end{keyword}
\received{5 May 2016}% Updated by VTEXPTS2LaTeX.exe, 03.06.2016 08:45
\revised{16 May 2016}% Updated by VTEXPTS2LaTeX.exe, 03.06.2016 08:45
\accepted{16 May 2016}% Updated by VTEXPTS2LaTeX.exe, 03.06.2016 08:45
\publishedonline{7 June 2016}
\end{frontmatter}

%s1 ###
\section{Introduction}\label{sec1}

Our purpose is to establish large deviations principle for the maximum
likelihood estimator of drift parameter of the Ornstein--Uhlenbeck
process driven by a mixed fractional Brownian motion:
%
%e1 ###
\begin{equation}
\label{eq101} dX_t=-\vartheta X_t dt+d
\widetilde{B}_t,\quad t\in[0,T],\ T>0,
\end{equation}
where the initial state $X_0=0$, and the drift parameter
$\vartheta$ is strictly positive. The process $\widetilde{B}$ is a
mixed fractional Brownian motion
\begin{equation*}
\widetilde{B}_t=B_t+B^H_t,\quad
t\in[0,T],
\end{equation*}
where $B=(B_t)$ is a Brownian motion, and $B^H=(B^H_t)$ is an
independent fractional Brownian motion with Hurst exponent $H\in(0,1]$,
that is, the centered Gaussian process with covariance function
\begin{equation*}
R(s,t)=\mathbb{E} B^H_t B^H_s=
\frac{1}{2} \bigl( t^{2H}+s^{2H}-\lvert t-s
\rvert^{2H} \bigr) ,\quad s,t\in[0,T].
\end{equation*}

It is important to notice that the parameter $H$ is considered to be
known. The problem of Hurst parameter estimation is not considered in
this work but presents a great interest for the future research.

Two following chapters contain information about maximum likelihood
estimation procedure for the mixed fractional Ornstein--Uhlenbeck
process and description of basic concepts of large deviations theory.

The formulation of our main results and their proofs are given in
Section~\ref{sec4}, whereas Section~\ref{sec5} contains auxiliary results.

%s2 ###
\section{Maximum likelihood estimation procedure}\label{sec2}

The interest to mixed fractional Brownian motion was triggered by
\xch{Cheridito}{Cheridito}~\cite{Cheridito}.
Further, a mixed fractional Brownian motion and related models were
comprehensively considered by Mishura \cite{M}. Finally, the results of
recent works of \xch{Cai, Kleptsyna, and Chigansky}{Kleptsyna and Chigansky} \cite{KlChCai} \xch{and Chigansky and Kleptsyna}{and} \cite{KlCh}
concerning the new canonical representation of mixed fractional
Brownian motion present a great value for the purposes of this paper.

An interesting change in properties of a mixed fractional Brownian
motion~$\widetilde{B}$ occurs depending on the value of $H$. In
particular, it was shown (see \cite{Cheridito}) that $\widetilde{B}$ is
a semimartingale in its own filtration if and only if either $H=\frac
{1}{2}$ or $H\in(\tfrac{3}{4},1]$.

The main contribution of paper \cite{KlChCai} is a novel approach to
the analysis of mixed fractional Brownian motion based on the filtering
theory of Gaussian processes. The core of this method is a new
canonical representation of $\widetilde{B}$.

In fact, there is an integral transformation that changes the mixed
fractional Brownian motion to a martingale. In particular (see \cite
{KlChCai}), let $g(s,t)$ be the solution of the following
integro-differential equation
%
%e2 ###
\begin{equation}
\label{gst} g(s,t)+H\frac{d}{ds}\int_0^t
g(r,t)\lvert s-r \rvert^{2H-1}\text {sign}(s-r)dr=1,\quad0<s<t\leq T.
\end{equation}
Then the process
%
%e3 ###
\begin{equation}
\label{eq104} M_t=\int_0^t g(s,t)d
\widetilde{B}_s, \quad t\in[0,T],
\end{equation}
is a Gaussian martingale with quadratic variation
\begin{equation*}
{\langle M\rangle}_t=\int_0^t
g(s,t)ds, \quad t\in[0,T].
\end{equation*}

Moreover, the natural filtration of the martingale $M$ coincides with
that of the mixed fractional Brownian motion $\widetilde{B}$.

Further, to what has just been mentioned concerning the mixed
fractional Brownian motion, an auxiliary semimartingale, appropriate
for the purposes of statistical analysis, can be also associated to the
corresponding Ornstein--Uhlenbeck process $X$ defined by (\ref
{eq101}). In particular, for the martingale $M$ defined by (\ref
{eq104}), the sample paths of the process $X$ are smooth enough in the
sense that the following process is well defined:
%
%e4 ###
\begin{equation}
\label{defQt} Q_t=\frac{d}{d{\langle M\rangle}_t} \int_0^t
g(s,t) X_s ds.
\end{equation}
We also define the process $Z=(Z_t, t \in[0,T])$ by
%
%e5 ###
\begin{equation}
\label{defZt} Z_t=\int_0^t
g(s,t)dX_s.
\end{equation}

One of the most important results of \cite{KlChCai} is that the process
$Z$ is the fundamental semimartingale associated to $X$ in the
following sense.

\begin{thm}\label{thm1}
Let $g(s,t)$ be the solution \xch{of}{of Eq.}~\eqref{gst}, and the process $Z$ be
defined by~\eqref{defZt}. Then the following assertions hold:
\begin{enumerate}

\item$Z$ is a semimartingale with the decomposition
%
%e6 ###
\begin{equation}
\label{eq107} Z_t=-\vartheta\int_0^t
Q_s d{\langle M\rangle}_s+M_t,
\end{equation}
where $M_t$ is the martingale defined by \eqref{eq104}.

\item$X$ admits the representation
%
%e7 ###
\begin{equation}
\label{Xinv} X_t=\int_0^t
\widehat{g}(s,t) dZ_s,
\end{equation}

where
\begin{equation*}
\widehat{g}(s,t)=1-\frac{d}{d{\langle M\rangle}_s} \int_0^t
g(r,s)dr.
\end{equation*}

\item The natural filtrations $(\mathcal{X}_t)$ and $(\mathcal{Z}_t)$
of $X$ and $Z$, respectively, coincide.
\end{enumerate}
\end{thm}

In addition, it was shown by \xch{Chigansky and Kleptsyna}{Kleptsyna and Chigansky} \cite{KlCh} that
the process~$Q$ admits the following representation:
%
%e8 ###
\begin{equation}
\label{Qtpsi} Q_t=\int_0^t
\psi(s,t)dZ_s=\frac{1}{2}\psi(t,t) Z_t+
\frac{1}{2}\int_0^t \psi(s,s)
dZ_s, \quad t \in[0,T],
\end{equation}
with
\begin{equation*}
\psi(s,t) = \frac{1}{2} \biggl( \frac{dt}{d{\langle M\rangle}_t}+\frac
{ds}{d{\langle M\rangle}_s}
\biggr).
\end{equation*}

The specific structure of the process $Q$ allows us to determine the
likelihood function for \eqref{eq101}, which according to Corollary
2.9 in \cite{KlChCai} equals
\begin{equation*}
L_T (\vartheta,X)=\frac{d\mu^X}{d\mu^{\widetilde{B}}}(X)=\exp{ \Biggl( -\vartheta\int
_0^T Q_t dZ_t-
\frac{1}{2} \vartheta^2 \int_0^T
Q_t^2 d{\langle M\rangle}_t \Biggr)},
\end{equation*}
where $\mu^X$ and $\mu^{\widetilde{B}}$ are the probability
measures induced by the processes $X$ and $\widetilde{B}$,
respectively. Thus, the score function for \eqref{eq101}, that is,
the derivative of the log-likelihood function from observations over
the interval $[0,T]$ is given by
\begin{equation*}
\varSigma_T(\theta)=-\int_0^T
Q_t dZ_t-\vartheta\int_0^T
Q_t^2 d{\langle M\rangle}_t,
\end{equation*}
which allows us to determine the maximum likelihood estimator
for the drift parameter $\vartheta$. Moreover, according to Theorem 2.9
in \cite{KlCh}, which is also presented further, the maximum likelihood
estimator is asymptotically normal.

\begin{thm}\label{thm2}
Let $g(s,t)$ be the solution \xch{of}{of Eq.}~\eqref{gst}, and let the
processes $Q$ and~$Z$ be defined by \eqref{defQt} and \eqref{defZt},
respectively. The maximum likelihood estimator of $\vartheta$ is given by
%
%e9 ###
\begin{equation}
\label{eq109} \widehat{\vartheta}_T(X)=-\frac{\int_0^T Q_t dZ_t}{\int_0^T Q_t^2
d{\langle M\rangle}_t}.
\end{equation}
Since $\vartheta>0$, this estimator is asymptotically normal
at the usual rate\textup{:}
\begin{equation*}
\sqrt{T} \bigl(\widehat{\vartheta}_T(X)-\vartheta \bigr)\xrightarrow
[T\rightarrow\infty]{d}N(0,2\vartheta).
\end{equation*}
\end{thm}

We will develop this result by proving the large deviation principle
for the max\-i\-mum-likelihood estimator \eqref{eq109}.

%s3 ###
\section{Large deviation principle}\label{sec3}

The large deviations principle characterizes the limiting behavior of a
family of random variables (or corresponding probability measures) in
terms of a rate function.

A rate function $I$ is a lower semicontinuous function $I:\mathbb
{R}\rightarrow[0,+\infty]$ such that, for all $\alpha\in[0,+\infty)$,
the level sets $\{x:I(x)\leq\alpha\}$ are closed subsets of $\mathbb
{R}$. Moreover,~$I$ is called a good rate function if its level sets
are compacts.

We say that a family of real random variables $ (Z_T  )_{T>0}$
satisfies the large deviation principle with rate function $I:\mathbb
{R}\rightarrow[0,+\infty]$ if for any Borel set $\varGamma\subset\mathbb{R}$,
\begin{equation*}
-\inf_{x\in\varGamma^o} I(x) \leq\liminf\limits
_{T\rightarrow\infty} \frac{1}{T} \log
\mathbb{P} (Z_T \in\varGamma ) \leq\limsup \limits
_{T\rightarrow\infty} \frac{1}{T}
\log\mathbb{P} (Z_T \in \varGamma ) \leq-\inf_{x\in\overline{\varGamma}}
I(x),
\end{equation*}
where $\varGamma^o$ and $\overline{\varGamma}$ denote the interior
and closure of $\varGamma$, respectively. Note that a family of random
variables can have at most one rate function associated with its large
deviation principle (for the proof, we refer the reader to the book by
Dembo and Zeitouni \cite{DZ}). Moreover, it is obvious that if $
(Z_T  )_{T>0}$ satisfies the large deviation principle and a Borel
set $\varGamma\subset\mathbb{R}$ is such that
\begin{equation*}
\inf_{x\in\varGamma^o} I(x) =\inf_{x\in\overline{\varGamma}} I(x),
\end{equation*}
then
\begin{equation*}
\lim\limits
_{T\rightarrow\infty} \frac{1}{T} \log\mathbb{P} (Z_T \in
\varGamma ) =-\inf_{x\in\varGamma} I(x).
\end{equation*}

We shall prove the large deviation principle for a family of maximum
likelihood estimators \eqref{eq109} via a similar approach as that of
\cite{Bercu1} and \cite{Bercu2} for an Ornstein--Uhlenbeck process and
fractional Ornstein--Uhlenbeck process, respectively.

In order to prove the large deviations principle for the drift
parameter estimator of mixed fractional Ornstein--Uhlenbeck process
\eqref{eq101}, the main tool is the normalized cumulant generating
function of arbitrary linear combination of $\int_0^T Q_t dZ_t$ and
$\int_0^T Q_t^2 d{\langle M\rangle}_t$,
%
%e10 ###
\begin{equation}
\label{Ltdef1} \mathcal{L}_T (a,b)=\frac{1}{T}\log\mathbb{E}
\bigl[ \exp \bigl(\mathcal {Z}_T (a,b) \bigr) \bigr],
\end{equation}
where, for any $(a,b)\in\mathbb{R}^2$,
\begin{equation*}
\mathcal{Z}_T (a,b)=a\int_0^T
Q_tdZ_t+b\int_0^T
Q^2_td {\langle M\rangle}_t.
\end{equation*}

Note that, for some $(a,b)\in\mathbb{R}^2$, the expectation in \eqref
{Ltdef1} may be infinite. In fact, in order to establish a large
deviation principle for $\widehat{\vartheta}_T$ it suffices to find the
limit of $\mathcal{L}_T(a,b)$ as $T \rightarrow\infty$ and apply the
following lemma, which is a consequence of \xch{the}{thw} G{\"a}rtner--Ellis
theorem (Theorem 2.3.6 in \cite{DZ}).

\begin{lemma}\label{lemma01}
For a family of maximum likelihood estimators $ (\vartheta_T
)_{T>0}$, let the function $\mathcal{L}_T (a,b)$ be defined by \eqref
{Ltdef1}, and, for each fixed value of $x$, let $\Delta_x$ denote the
set of $a\in\mathbb{R}$ for which $\lim_{T\to\infty}\mathcal
{L}_T(a,-xa)$ exists and is finite. If $\Delta_x$ is not empty for each
value of $x$, then $ (\vartheta_T  )_{T>0}$ satisfies the
large deviation principle with a good rate function
%
%e11 ###
\begin{equation}
\label{defIx} I(x)=-\inf_{a\in\Delta_x} \lim_{T\to\infty}
\mathcal{L}_T(a,-xa).
\end{equation}
\end{lemma}

%s4 ###
\section{Main results}\label{sec4}

\begin{thm}\label{th02}
The maximum likelihood estimator $\widehat{\vartheta}_T$ defined by
\eqref{eq109} satisfies the large deviation principle with the good
rate function
\begin{equation*}
I(x)= %
\begin{cases}
-\frac{(x+\vartheta)^2}{4x} & \mbox{if } x<-\frac{\vartheta}{3},\\%[12pt]
2x+\vartheta& \mbox{if } x\geq-\frac{\vartheta}{3}.
\end{cases} %
\end{equation*}
\end{thm}

\begin{proof}
As it was mentioned in the previous section, in order to
establish the large deviation principle for $\widehat{\vartheta}_T$ and
determine the corresponding good rate function, it is necessary to find
the limit
%
%e12 ###
\begin{equation}
\label{Lab} \mathcal{L}(a,b)=\lim_{T\to\infty}\mathcal{L}_T(a,b)
\end{equation}
and determine the set of $(a,b) \in\mathbb{R}^2$ for which
this limit is finite.

For arbitrary $\varphi\in\mathbb{R}$, consider the Doleans
exponential of $(\varphi+\vartheta)\int_0^t Q_s dM_s$,
\begin{equation*}
\varLambda_{\varphi}(t)=\exp \Biggl((\varphi+\vartheta)\int
_0^t Q_s dM_s-
\frac{(\varphi+\vartheta)^2}{2}\int_0^t Q_s^2
d{\langle M\rangle }_s \Biggr).
\end{equation*}

Note that $ ( \frac{1}{\sqrt{\psi(t,t)}}Q_t  )_{t \geq0}$ is
a Gaussian process whose mean and variance functions are bounded on
$[0,T]$. Thus, $\varLambda_{\varphi}$ satisfies the conditions of
Girsanov's theorem in accordance with Example 3 of paragraph 2 of
Section~6 in \cite{LS}, and we can apply a usual change of measures and
consider the new probability $\mathbb{P}_{\varphi}$ defined by the
local density
\begin{equation*}
\frac{d\mathbb{P}_{\varphi}}{d\mathbb{P}}=\varLambda_{\varphi}(T)=\exp \Biggl((\varphi+\vartheta)\int
_0^T Q_t dM_t-
\frac{(\varphi+\vartheta
)^2}{2}\int_0^T Q_t^2
d{\langle M\rangle}_t \Biggr).
\end{equation*}

Observe that, due to \eqref{eq107}, $\varLambda_{\varphi}(T)$ can be
rewritten in terms of the fundamental semimartingale
\begin{align*}
\frac{d\mathbb{P}_{\varphi}}{d\mathbb{P}} & {}=\varLambda_{\varphi}(T)                                                                                                                                                                           \\
                                          & {}=\exp \Biggl(\!(\varphi+\vartheta)\!\int_0^T\!\! Q_t dZ_t\,{+}\,(\varphi+\vartheta )\vartheta\!\int_0^T\!\! Q_t^2 d{\langle M \rangle}_t \,{-}\,\frac{(\varphi +\vartheta)^2}{2}\!\int_0^T\!\! Q_t^2 d{\langle M\rangle}_t\! \Biggr) \\
                                          & {}=\exp \Biggl(\!(\varphi+\vartheta)\!\int_0^T\!\! Q_t dZ_t-\frac{\varphi ^2-\vartheta^2}{2}\!\int_0^T\!\! Q_t^2 d{\langle M\rangle}_t \Biggr).
\end{align*}
Consequently, we can rewrite $\mathcal{L}_T (a,b)$ as
\begin{align*}
\mathcal{L}_T (a,b) & {}=\frac{1}{T}\log\mathbb{E} \bigl[ \exp \bigl(\mathcal {Z}_T (a,b) \bigr) \bigr]                                                                                                            \\
                    & {}=\frac{1}{T} \log \mathbb{E}_{\varphi} \bigl[ \exp \bigl(\mathcal{Z}_T (a,b) \bigr) \varLambda_{\varphi}(T)^{-1} \bigr]                                                                    \\
                    & {}=\frac{1}{T} \log\mathbb{E}_{\varphi} \exp \Biggl(\! (a\,{-}\,\varphi\,{-}\, \vartheta )\!\int_0^T\!\! Q_t dZ_t\,{+}\,\frac{1}{2}\bigl(2b\,{-}\,\vartheta^2\,{+}\, \varphi^2\bigr) \!\int_0^T\!\! Q_t^2 d{\langle M\rangle}_t\! \Biggr).
\end{align*}
Given an arbitrary real number $\varphi$, we can choose
$\varphi=a-\vartheta$. Then
\begin{equation*}
\mathcal{L}_T (a,b)=\frac{1}{T} \log\mathbb{E}_{\varphi}
\exp \Biggl( \frac{1}{2}\bigl(2b-\vartheta^2+(a-
\vartheta)^2\bigr) \int_0^T
Q_t^2 d{\langle M\rangle}_t \Biggr)
\end{equation*}
or, denoting $\mu=-\frac{1}{2}(2b-\vartheta^2+(a-\vartheta)^2)$,
%
%e13 ###
\begin{equation}
\label{LTmu} \mathcal{L}_T (a,b)=\frac{1}{T} \log
\mathbb{E}_{\varphi} \exp \Biggl( - \mu\int_0^T
Q_t^2 d{\langle M\rangle}_t \Biggr).
\end{equation}
As it was mentioned before, the expectation \xch{in}{in Eq.}~\eqref
{LTmu} can be infinite for some combinations of $\mu$ and $\varphi$.
Our purpose is to determine the set of $(\varphi,\mu)\in\mathbb{R}^2$
for which this expectation and limit \eqref{Lab} are finite.
According to Girsanov's theorem, under $\mathbb{P}_{\varphi}$, the process
%
%e14 ###
\begin{equation}
\label{equa207} M_t-(\varphi+\vartheta)\int_0^t
Q_s d{\langle M\rangle}_s=Z_t-\varphi \int
_0^t Q_s d{\langle M
\rangle}_s
\end{equation}
has the same distribution as $M$ under $\mathbb{P}$. Consequently,
applying the inverse integral transformation \eqref{Xinv} to
\eqref{equa207}, we get that, under $\mathbb{P}_{\varphi }$, the
process $X_t-\varphi\int_0^t X_s ds$ is a mixed fractional Brownian
motion.

Under the new probability measure $\mathbb{P}_{\varphi}$, the process
$X$ is a mixed fractional Orn\-stein--Uhlenbeck process with drift
parameter $-\varphi$. Consequently, in order to find the limit of \eqref{LTmu}
as $T\to\infty$, we can apply Lemma \ref{lemma01}, which is presented in
Section~\ref{sec5}. Thus, we have the equality
\begin{equation*}
\mathcal{L}(a,b)= -\frac{\varphi}{2}-\sqrt{\frac{\varphi^2}{4}+
\frac{\mu
}{2}}=-\frac{1}{2} \bigl(a-\vartheta+\sqrt{
\vartheta^2-2b} \bigr),
\end{equation*}
and convergence \eqref{Lab} holds for $\mu
>-\frac{\varphi^2}{2}$, which gives $\vartheta
^2-2b>0$.

For $x \in\mathbb{R}$, denote the function
\begin{equation*}
L_x(a)=\mathcal{L}(a,-xa)=-\frac{1}{2} \bigl(a-\vartheta+
\sqrt{\vartheta ^2+2xa} \bigr)
\end{equation*}
defined on the set
\begin{equation*}
\Delta_x= \bigl\{ a \in\mathbb{R} | \vartheta ^2+2xa>0 \bigr\}.
\end{equation*}
Then, according to \eqref{defIx}, the rate function $I(x)$
for the maximum likelihood estimator $\widehat{\vartheta}_T$ can be
found as
\begin{equation*}
I(x)=-\inf_{a\in\Delta_x} L_x(a).
\end{equation*}
Consequently, straightforward calculations of this infimum
finish the proof of the theorem.
\end{proof}

\begin{remark}\label{rem1}
Observe that the rate function $I(x)$ does not depend on the parameter
$H$. Hence, $\widehat{\vartheta}_T$ shares the same large deviation
principles as those established by Florens-Landais and Pham \cite{FLP}
for a standard Ornstein--Uhlenbeck process and by Bercu, Coutin, and
Savy \cite{Bercu2} for a fractional Ornstein--Uhlenbeck process
(see also \xch{\cite{Bercu3,Gamboa}}{\cite{Bercu3},\cite{Gamboa}}).
\end{remark}

%s5 ###
\section{Auxiliary results}\label{sec5}

We can observe that the following lemma plays a key role in the proof
of Theorem \ref{th02}.

\begin{lemma}\label{lemma02}
For a mixed fractional Ornstein--Uhlenbeck process $X$ with
drift parameter $\vartheta$, we have the following limit:
%
%e15 ###
\begin{equation}
\label{eq201} \mathcal{K}_T(\mu)=\frac{1}{T}\log\mathbb{E}
\exp \Biggl(-\mu\int_0^T Q_t^2
d{\langle M\rangle}_t \Biggr) \rightarrow\frac{\vartheta}{2}-\sqrt {
\frac{\vartheta^2}{4}+\frac{\mu}{2}}, \quad T \rightarrow\infty,
\end{equation}
for all $\mu>-\frac{\vartheta^2}{2}$.
\end{lemma}

\begin{proof}
We shall prove the lemma using an approach similar to that in
\cite{KlCh}. Denote $V_t=\int_0^t \psi(s,s) dZ_s$. Then, according to
\eqref{Qtpsi}, we can rewrite
\begin{align*}
dZ_t&{}=-\vartheta Q_t d{\langle M\rangle}_t+dM_t=-
\frac{\vartheta}{2} \psi(t,t) Z_t d{\langle M\rangle}_t-
\frac{\vartheta}{2} V_t d{\langle M\rangle}_t+dM_t\\
&{}=-\frac{\vartheta}{2} Z_t dt-\frac{\vartheta}{2} V_t
\frac{1}{\psi
(t,t)} dt+\frac{1}{\sqrt{\psi(t,t)}}dW_t,
\end{align*}
where $W_t$ is a Brownian motion. Consequently, we get
\begin{equation*}
dV_t=\psi(t,t) dZ_t=-\frac{\vartheta}{2} \psi(t,t)
Z_t dt-\frac
{\vartheta}{2} V_t dt+\sqrt{
\psi(t,t)}dW_t.
\end{equation*}

The Gaussian vector $\zeta_t= (Z_t,V_t )^T$ is a solution of
the linear system of the It\^o stochastic differential equation
\begin{equation*}
d\zeta_t=-\frac{\vartheta}{2}A(t)\zeta_t dt +
b(t)dW_t,
\end{equation*}
where
\begin{align*}
A(t)= %
\begin{pmatrix}
1 & \frac{1}{\psi(t,t)}\\%[12pt]
\psi(t,t) & 1
\end{pmatrix} %
\quad\text{and} \quad b(t)=
\begin{pmatrix}
\frac{1}{\sqrt{\psi(t,t)}} \\[3pt]
\sqrt{\psi(t,t)}
\end{pmatrix} %
.
\end{align*}
Moreover, $\mathcal{K}_T(\mu)$ in \eqref{eq201} can be
rewritten as
\begin{align*}
\mathcal{K}_T(\mu)&{}=\frac{1}{T}\log\mathbb{E} \exp \Biggl(-\mu
\int_0^T Q_t^2 d{\langle
M\rangle}_t \Biggr)
\\
&{}=\frac{1}{T}\log\mathbb{E} \exp \Biggl(-\frac{\mu}{4} \int
_0^T \bigl(\psi (t,t) Z_t+V_t
\bigr)^2 d{\langle M\rangle}_t \Biggr)
\\
&{}=\frac{1}{T}\log\mathbb{E} \exp \Biggl(-\frac{\mu}{4} \int
_0^T \biggl(\sqrt{\psi(t,t)} Z_t+
\frac{1}{\sqrt{\psi(t,t)}} V_t \biggr)^2 dt \Biggr)
\\
&{}=\frac{1}{T}\log\mathbb{E} \exp \Biggl(-\frac{\mu}{4} \int
_0^T \zeta _t^T R(t)
\zeta_t dt \Biggr),
\end{align*}
where
\begin{align*}
R(t)= %
\begin{pmatrix}
\psi(t,t) & 1\\%[12pt]
1 & \frac{1}{\psi(t,t)}
\end{pmatrix} %
.
\end{align*}
By the Cameron--Martin-type formula from Section~4.1 of \cite{KlBr},
\begin{equation*}
\mathcal{K}_T(\mu)=-\frac{\mu}{4T} \int_0^T
\text{tr} \bigl(\varGamma(t) R(t) \bigr) dt,
\end{equation*}
where $\varGamma(t)$ is the solution of the equation
%
%e16 ###
\begin{equation}
\label{eq202} \dot{\varGamma(t)}=-\frac{\vartheta}{2} A(t) \varGamma(t)-
\frac{\vartheta
}{2}\varGamma(t) A^T(t)-\frac{\mu}{2}\varGamma(t)R(t)
\varGamma(t)+B(t)
\end{equation}
with $B(t)=b(t)b^T(t)$ and initial condition $\varGamma(0)=0$.

We shall search solution \xch{of}{of Eq.}~\eqref{eq202} as the ratio $\varGamma
(t)=\varPsi^{-1}_1(t)\varPsi_2(t)$, where $\varPsi_1(t)$ and $\varPsi_2(t)$ are the
solutions of the following equation system:
%
%e17 ###
\begin{equation}
\label{eq203} %
\begin{aligned}
\dot{\varPsi}_1 (t)&{}= \frac{\vartheta}{2}\varPsi_1(t)A(t)+\frac{\mu}{2}\varPsi _2(t)R(t), \\
\dot{\varPsi}_2 (t)&{}=\varPsi_1(t)B(t)-\frac{\vartheta}{2}\varPsi_2(t)A^T (t),
\end{aligned} %
\end{equation}
with initial conditions $\varPsi_1 (0)=I$ and $\varPsi_2 (0)=0$.
From the first equation of \eqref{eq203} we get
\begin{equation*}
\varPsi^{-1}_1(t) \dot{\varPsi}_1 (t)=
\frac{\vartheta}{2}A(t)+\frac{\mu
}{2}\varGamma(t)R(t),
\end{equation*}
and since $\text{tr }A(t)=2$, it follows that
\begin{equation*}
\frac{\mu}{2} \text{tr}\bigl(\varGamma(t)R(t)\bigr)=\text{tr}\bigl(
\varPsi^{-1}_1(t) \dot {\varPsi}_1 (t)\bigr)-
\vartheta.
\end{equation*}
Since $\dot{\varPsi}_1 (t)=\varPsi_1 (t)(\varPsi_1^{-1} (t)\dot{\varPsi}_1 (t))$,
by Liouville's formula we have
\begin{align*}
-\frac{\mu}{4T} \int_0^T \text{tr} \bigl(
\varGamma(t) R(t) \bigr) dt&{}=-\frac
{1}{2T}\int_0^T
\text{tr }\bigl(\varPsi^{-1}_1(t) \dot{\varPsi}_1
(t)\bigr) dt + \frac
{\vartheta}{2}
\\
&{}=-\frac{1}{2T}\log\det\varPsi_1 (T)+\frac{\vartheta}{2}.
\end{align*}
In order to calculate $\lim_{T\rightarrow\infty}\frac
{1}{T}\log\det\varPsi_1 (T)$, define the matrix
\begin{align*}
J= %
\begin{pmatrix}
0 & 1\\%[12pt]
1 & 0
\end{pmatrix} %
\end{align*}
and note that $A(t)^T=JA(t)J$, $R(t)=JA(t)$, and
$B(t)=A(t)J$. Setting $\widetilde{\varPsi}_2(t)=\varPsi_2(t)J$, from \eqref{eq203}
we obtain the following equation system:
%
%e18 ###
\begin{equation}
\label{eq204} %
\begin{aligned} \dot{\varPsi}_1 (t)&{}=
\frac{\vartheta}{2}\varPsi_1(t)A(t)+\frac{\mu}{2}\varPsi
_2(t)A(t),
\\
\dot{\widetilde{\varPsi}}_2 (t)&{}=\varPsi_1(t)A(t)-
\frac{\vartheta}{2}\widetilde {\varPsi}_2(t)A(t),
\end{aligned}
\end{equation}
with initial conditions $\varPsi_1 (0)=I$ and $\widetilde{\varPsi
}_2 (0)=0$. When $\frac{\vartheta^2}{2}+\mu>0$, the coefficient matrix
of system \eqref{eq204}
\begin{align*}
\begin{pmatrix}
\frac{\vartheta}{2} & \frac{\mu}{2}\\[3pt]
1 & -\frac{\vartheta}{2}
\end{pmatrix} %
\end{align*}
has two real eigenvalues $\pm\lambda$ with $\lambda=\sqrt
{\frac{\vartheta^2}{4}+\frac{\mu}{2}}$ and eigenvectors
\begin{align*}
v^{\pm}= %
\begin{pmatrix}
\frac{\vartheta}{2}\pm\lambda\\%[12pt]
1
\end{pmatrix} %
.
\end{align*}
Denote $a^{\pm}=\frac{\vartheta}{2}\pm\lambda=\frac
{\vartheta}{2}\pm\sqrt{\frac{\vartheta^2}{4}+\frac{\mu}{2}}$.
Diagonalizing system \eqref{eq204}, we get
\begin{equation*}
\varPsi_1 (t)=a^{+}\varUpsilon_1(t)+a^{-}
\varUpsilon_2(t),
\end{equation*}
where $\varUpsilon_1(t)$ and $\varUpsilon_2(t)$ are the solutions
of the equations
%
%e19 ###
\begin{equation}
\label{eq205} %
\begin{aligned} \dot{\varUpsilon}_1 (t)=
\lambda\varUpsilon_1(t) A(t),
\\
\dot{\varUpsilon}_2 (t)=-\lambda\varUpsilon_2(t) A(t),
\end{aligned} %
\end{equation}
with initial conditions $\varUpsilon_1(0)=\frac{1}{2\lambda}I$
and $\varUpsilon_2(0)=-\frac{1}{2\lambda}I$. Denote the matrix
$M(T)=\varUpsilon_2^{-1}(T)\varUpsilon_1(T)$, which is the solution of the equation
%
%e20 ###
\begin{equation}
\label{eq206} \dot{M}(t)=\lambda \bigl( A(t)M(t)+M(t)A(t) \bigr)
\end{equation}
subject to initial condition $M(0)=-I$. Then
\begin{align*}
\frac{1}{T}\log\det\varPsi_1(T) & {}=\frac{1}{T}\log\det \bigl(a^{+}\varUpsilon _1(T)+a^{-} \varUpsilon_2(T) \bigr)                                                                                                      \\
                                & {}=\frac{1}{T}\log\det \bigl(a^{-}\varUpsilon_2(T) \bigr)+\frac{1}{T}\log \det \biggl(I+\frac{a^{+}}{a^{-}}M(T) \biggr)                                                                \\
                                & {}=\frac{1}{T}\log\det \bigl(a^{-}\varUpsilon_2(T) \bigr)\\
                                &\quad{}+\frac{1}{T}\log \biggl( 1+{ \biggl(\frac{a^{+}}{a^{-}} \biggr)}^{2}\det M(T) +\frac {a^{+}}{a^{-}} \text{tr } {M(T)} \biggr) \\
                                & {}=\frac{1}{T}\log\det \bigl(a^{-}\varUpsilon_2(T) \bigr)+\frac{1}{T}\log \biggl( 1+{ \biggl(\frac{a^{+}}{a^{-}} \biggr)}^{2}\det M(T) \biggr)                                         \\
                                &\quad{}+\frac{1}{T}\log \biggl( 1+\frac{\frac{a^{+}}{a^{-}} \text{tr } {M(T)}}{1+{ (\frac{a^{+}}{a^{-}} )}^{2}\det M(T)} \biggr).
\end{align*}
Applying Liouville's formula to \eqref{eq205}, we get
\begin{align*}
&{}\frac{1}{T}\log\det \bigl(a^{-}\varUpsilon_2(t)
\bigr)+\frac{1}{T}\log \biggl( 1+{ \biggl(\frac{a^{+}}{a^{-}}
\biggr)}^{2}\det M(T) \biggr)
\\
&\quad{}=\frac{1}{T}\log \biggl({ \biggl(\frac{a^{-}}{2\lambda} \biggr)}^{2}
\exp (-2\lambda T) \biggr)+\frac{1}{T}\log \biggl( 1+{ \biggl(
\frac
{a^{+}}{a^{-}} \biggr)}^{2}\exp(4\lambda T) \biggr)\rightarrow2
\lambda
\end{align*}
as $T\rightarrow\infty$. Thus, in order to prove that limit
\eqref{eq201} holds, we should show that
%
%e21 ###
\begin{equation}
\label{eq2007} \frac{1}{T}\log \biggl( 1+\frac{\frac{a^{+}}{a^{-}} \text{tr }
{M(T)}}{1+{ (\frac{a^{+}}{a^{-}} )}^{2}\exp(4\lambda T)} \biggr)
\rightarrow0,\quad T\rightarrow\infty.
\end{equation}
Given \eqref{eq206}, by Theorem 3 in \cite{Vidyasagar} we have
\begin{equation*}
| \text{tr } {M(T)} | \leq2 \sqrt{2} \exp(2 \lambda T).
\end{equation*}
Thus, limit \eqref{eq2007} holds, which finishes the proof
of the lemma.
\end{proof}

\section*{Acknowledgments}

My deepest gratitude is to my advisor Marina Kleptsyna, who proposed me
to consider the problem presented in this paper and whose advices and
help were very important. I also would like to thank the reviewers
whose remarks and advices allowed me to significantly improve this article.

% structpyb loaded by ispudulyte, 2016-06-03 08:52:21
%


\begin{thebibliography}{13}

%b1 ###bbsrt2
%b3 ###
\bibitem{Bercu1}
\begin{barticle}
\bauthor{\bsnm{Bercu}, \binits{B.}},
\bauthor{\bsnm{Rouault}, \binits{A.}}:
\batitle{Sharp large deviations for the Ornstein--Uhlenbeck process}.
\bjtitle{Theory Probab. Appl.}
\bvolume{46}(\bissue{1}),
\bfpage{1}--\blpage{19}
(\byear{2002}).
\bid{doi={10.1137/S0040585X97978737}, mr={1968706}}
\end{barticle}
\OrigBibText
B. Bercu and A. Rouault.
\textit{Sharp large deviations for the Ornstein--Uhlenbeck process}.
Theory Probab. Appl., 46(1):1--19, 2002.
\endOrigBibText
\bptok{structpyb}
\endbibitem

%b2 ###bbsrt2
%b5 ###
\bibitem{Bercu3}
\begin{barticle}
\bauthor{\bsnm{Bercu}, \binits{B.}},
\bauthor{\bsnm{Gamboa}, \binits{F.}},
\bauthor{\bsnm{Rouault}, \binits{A.}}:
\batitle{Large deviations for quadratic forms of stationary Gaussian processes}.
\bjtitle{Stoch. Process. Appl.}
\bvolume{71}(\bissue{1}),
\bfpage{75}--\blpage{90}
(\byear{1997}).
\bid{doi={10.\\1016/S0304-4149(97)00071-9}, mr={1480640}}
\end{barticle}
\OrigBibText
B. Bercu, F. Gamboa, and A. Rouault.
\textit{Large deviations for quadratic forms of stationary Gaussian processes.}.
Stochastic Process Appl., 71(1):75--90, 1997.
\endOrigBibText
\bptok{structpyb}
\endbibitem

%%% bbsrt2.pl, ver. 2.5.5, 2015.06.11
%b3 ###bbsrt2
%b4 ###
\bibitem{Bercu2}
\begin{barticle}
\bauthor{\bsnm{Bercu}, \binits{B.}},
\bauthor{\bsnm{Coutin}, \binits{L.}},
\bauthor{\bsnm{Savy}, \binits{N.}}:
\batitle{Sharp large deviations for the fractional Ornstein--Uhlenbeck process}.
\bjtitle{Theory Probab. Appl.}
\bvolume{55}(\bissue{4}),
\bfpage{575}--\blpage{610}
(\byear{2010}).
\bid{doi={10.\\1137/S0040585X97985108}, mr={2859161}}
\end{barticle}
\OrigBibText
B. Bercu, L. Coutin, and N. Savy.
\textit{Sharp large deviations for the fractional Ornstein--Uhlenbeck process}.
Theory Probab. Appl., 55(4), 575--610, 2010.
\endOrigBibText
\bptok{structpyb}
\endbibitem


%b4 ###bbsrt2
%b1 ###
\bibitem{KlChCai}
\begin{botherref}
\oauthor{\bsnm{Cai}, \binits{C.}},
\oauthor{\bsnm{Chigansky}, \binits{P.}},
\oauthor{\bsnm{Kleptsyna}, \binits{M.}}:
Mixed Gaussian processes: A filtering approach.
Ann. Probab.
(2015, to appear).
\arxivurl{1208.6253}
\end{botherref}
\OrigBibText
C. Cai, P. Chigansky, and M. Kleptsyna.
\textit{Mixed Gaussian processes: A filtering approach.}.
to appear in Ann. Probab., arXiv preprint 1208.6253, 2015.
\endOrigBibText
\bptok{structpyb}
\endbibitem

%b5 ###bbsrt2
%b9 ###
\bibitem{Cheridito}
\begin{barticle}
\bauthor{\bsnm{Cheridito}, \binits{P.}}:
\batitle{Mixed fractional Brownian motion}.
\bjtitle{Bernoulli}
\bvolume{7}(\bissue{6}),
\bfpage{913}--\blpage{934}
(\byear{2001}).
\bid{doi={10.2307/3318626}, mr={1873835}}
\end{barticle}
\OrigBibText
P. Cheridito.
\textit{Mixed fractional Brownian motion.}
Bernoulli, 7(6):913--934, 2001.
\endOrigBibText
\bptok{structpyb}
\endbibitem

%b6 ###bbsrt2
%b2 ###
\bibitem{KlCh}
\begin{botherref}
\oauthor{\bsnm{Chigansky}, \binits{P.}},
\oauthor{\bsnm{Kleptsyna}, \binits{M.}}:
Spectral asymptotics of the fractional Brownian motion covariance operator.
\arxivurl{1507.04194} (2015)
\end{botherref}
\OrigBibText
P. Chigansky and M. Kleptsyna.
\textit{Spectral asymptotics of the fractional Brownian motion
covariance operator}.
arXiv preprint 1507.04194, 2015.
\endOrigBibText
\bptok{structpyb}
\endbibitem

%b7 ###bbsrt2
%b13 ###
\bibitem{DZ}
\begin{bbook}
\bauthor{\bsnm{Dembo}, \binits{A.}},
\bauthor{\bsnm{Zeitouni}, \binits{O.}}:
\bbtitle{Large Deviations Techniques and Applications},
\bedition{2}nd edn.
\bsertitle{Applications of Mathematics (New York)},
vol.~\bseriesno{38}.
\bpublisher{Springer},
\blocation{New York}
(\byear{1998}).
\bid{doi={10.1007/978-1-4612-5320-4}, mr={1619036}}
\end{bbook}
\OrigBibText
A. Dembo and O. Zeitouni.
\textit{Large Deviations Techniques and Applications,} volume 38 of
\textit{Applications of Mathematics (New York)}.
Springer-Verlag, New York, second edition, 1998.
\endOrigBibText
\bptok{structpyb}
\endbibitem

%b8 ###bbsrt2
%b8 ###
\bibitem{FLP}
\begin{barticle}
\bauthor{\bsnm{\xch{Florens-Landais}{Florence-Landais}}, \binits{D.}},
\bauthor{\bsnm{Pham}, \binits{H.}}:
\batitle{Large deviations in estimation of an Ornstein--Uhlenbeck model}.
\bjtitle{J. Appl. Probab.}
\bvolume{36}(\bissue{1}),
\bfpage{60}--\blpage{77}
(\byear{1999}).
\bid{doi={10.1239/jap/\\1032374229}, mr={1699608}}
\end{barticle}
\OrigBibText
D. Florence-Landais and H. Pham.
\textit{Large deviations in estimation of an Ornstein--Uhlenbeck model.}
J. Appl. Probab., 36(1):60--77, 1999.
\endOrigBibText
\bptok{structpyb}
\endbibitem

%b9 ###bbsrt2
%b6 ###
\bibitem{Gamboa}
\begin{barticle}
\bauthor{\bsnm{Gamboa}, \binits{F.}},
\bauthor{\bsnm{Rouault}, \binits{A.}},
\bauthor{\bsnm{Zani}, \binits{M.}}:
\batitle{A functional large deviations principle for quadratic forms of
Gaussian stationary processes}.
\bjtitle{Stat. Probab. Lett.}
\bvolume{43}(\bissue{3}),
\bfpage{299}--\blpage{308}
(\byear{1999}).
\bid{doi={10.1016/S0167-7152(98)00270-3}, mr={1708097}}
\end{barticle}
\OrigBibText
F. Gamboa, A. Rouault, and M.Zani.
\textit{A functional large deviations principle for quadratic forms of
Gaussian stationary processes.}.
Statist. Probab. Lett., 43(3):299--308, 1999.
\endOrigBibText
\bptok{structpyb}
\endbibitem

%b10 ###bbsrt2
%b7 ###
\bibitem{KlBr}
\begin{barticle}
\bauthor{\bsnm{Kleptsyna}, \binits{M.L.}},
\bauthor{\bparticle{Le} \bsnm{Breton}, \binits{A.}}:
\batitle{Optimal linear filtering of general multidimensional Gaussian
processes and its application to Laplace transforms of quadratic functionals}.
\bjtitle{J. Appl. Math. Stoch. Anal.}
\bvolume{14}(\bissue{3}),
\bfpage{215}--\blpage{226}
(\byear{2001}).
\bid{doi={10.1155/\\S104895330100017X}, mr={1853078}}
\end{barticle}
\OrigBibText
M.L. Kleptsyna and A. Le Breton.
\textit{Optimal linear filtering of general multidimensional Gaussian
processes and its application to Laplace transforms of quadratic functionals.}
J. Appl. Math. Stochastic Anal., 14(3):215--226, 2001.
\endOrigBibText
\bptok{structpyb}
\endbibitem

%b11 ###bbsrt2
%b11 ###
\bibitem{LS}
\begin{bbook}
\bauthor{\bsnm{Liptser}, \binits{R.}},
\bauthor{\bsnm{Shiryaev}, \binits{A.}}:
\bbtitle{Statistics of Random Processes}.
\bpublisher{Nauka},
\blocation{Moscow}
(\byear{1974})
\end{bbook}
\OrigBibText
R. Liptser and A. Shiryaev.
\textit{Statistics of Random Processes,} Moscow, 1974.
\endOrigBibText
\bptok{structpyb}
\endbibitem

%b12 ###bbsrt2
%b10 ###
\bibitem{M}
\begin{bbook}
\bauthor{\bsnm{Mishura}, \binits{Y.}}:
\bbtitle{Stochastic Calculus for Fractional Brownian Motion and Related
Processes}.
\bpublisher{Springer},
\blocation{Berlin Heidelberg}
(\byear{2008}).
\bid{doi={10.1007/978-3-540-75873-0}, mr={2378138}}
\end{bbook}
\OrigBibText
Y. Mishura
\textit{Stochastic Calculus for Fractional Brownian Motion and Related
Processes,} Springer-Verlag Berlin Heidelberg, 2008.
\endOrigBibText
\bptok{structpyb}
\endbibitem

%b13 ###bbsrt2
%b12 ###
\bibitem{Vidyasagar}
\begin{bbook}
\bauthor{\bsnm{Vidyasagar}, \binits{M.}}:
\bbtitle{Nonlinear Systems Analysis},
\bedition{2}nd edn.
\bpublisher{SIAM},
\blocation{Philadelphia, PA}
(\byear{2002}).
\bid{doi={10.1137/1.9780898719185}, mr={1946479}}
\end{bbook}
\OrigBibText
M. Vidyasagar.
\textit{Nonlinear Systems Analysis, 2 ed.}
SIAM, 2002.
\endOrigBibText
\bptok{structpyb}
\endbibitem



\end{thebibliography}
\end{document}